\documentclass[10pt,a4paper]{amsart}
\usepackage{amsmath,amssymb,latexsym,tabularx}
\usepackage[dvips]{graphicx}
\headsep=1.2500cm \numberwithin{equation}{section}
\newtheorem{theorem}{Theorem}[section]
\newtheorem{proposition}[theorem]{Proposition}
\newtheorem{corollary}[theorem]{Corollary}
\newtheorem{lemma}[theorem]{Lemma}
\theoremstyle{definition}
\newtheorem{example}[theorem]{Example}
\newtheorem{definition}[theorem]{Definition}
\theoremstyle{remark}

\title[Asymptotic Resemblance]{ Asymptotic Resemblance}

\author[Sh. Kalantari]{Sh. Kalantari}
 \address[Sh. Kalantari]{Mathematics and Computer Science Department,
 Amirkabir University of Technology, 424 Hafez Avenue, 15914 Tehran,
Iran.}
 \email{shahab.kalantari@aut.ac.ir}

\author[B.~Honari]{B. Honari$^\dag$}

\address[B.~Honari]{Mathematics and Computer Science Department,
 Amirkabir University of Technology, 424 Hafez Avenue, 15914 Tehran,
Iran.}

\address[B.~Honari]{$^\dag$ Corresponding author.}
\email{honari@aut.ac.ir}
\date{Oct, 25, 2014}

\begin{document}
\maketitle
\begin{abstract}
\emph{Uniformity} and \emph{proximity} are two different ways for defining small scale structures on a set.\emph{ Coarse structures} are large scale counterparts of uniform structures. In this paper, motivated by the definition of proximity, we develop the concept of asymptotic resemblance as a relation between subsets of a set to define a large scale structure on it. We use our notion of asymptotic resemblance to generalize some basic concepts of coarse geometry. We introduce a large scale compactification which in special cases agrees with the \emph{Higson compactification}. At the end we show that how the \emph{asymptotic dimension} of a metric space can be generalized to a set equipped with an asymptotic resemblance relation.
\end{abstract}
\section{Introduction and Preliminaries}
There are several ways to define \emph{small scale} structures on a set. In 1937 Weil \cite{Wei} defined the concept of \emph{uniformity}. Few years later Tukey \cite{Tuk} used the notion of \emph{uniform coverings} to find another definition for uniform spaces. In 1950 Efremovich \cite{Ef1,Ef2} used \emph{proximity relations} to define a small scale structure on a set. He axiomatized the relation "$A$ \emph{is near} $B$" for subsets $A$ and $B$ of a set. Let us recall the definition of a proximity space.
\begin{definition}\label{prox}
A relation $\delta$ on the family of all subsets of a nonempty set $X$ is called a \emph{proximity} on $X$ if for all $A,B,C\subseteq X$, it satisfies the following properties (By $A\bar{\delta} B$ we mean that $A\delta B$ does not hold.)\\
i) If $A\delta B$ then $B\delta A$.\\
ii) $\emptyset \bar{\delta} A$.\\
iii) If $A\bigcap B\neq \emptyset$ then $A\delta B$.\\
iv) $A\delta(B\bigcup C)$ if and only if $A\delta B$ or $A\delta C$.\\
v) If $A\bar{\delta} B$ then there is $E\subseteq X$ such that $A\bar{\delta} E$ and $(X\setminus E)\bar{\delta} B$.\\
The pair $(X,\delta)$ is called a proximity space.
\end{definition}
There are also some ways to define \emph{large scale} structures on a set. In recent contexts one can find notions of \emph{coarse structures} \cite{Roe},\emph{ large scale structures} \cite{Alt} and \emph{ball structures} \cite{Nor}. A coarse structure $\mathcal{E}$ on a set $X$ is a family of subsets of $X\times X$, such that all subsets of a member of $\mathcal{E}$ are members of $\mathcal{E}$ and for all $E,F\in \mathcal{E}$ the sets $E^{-1}$, $E\circ F$ and $E\bigcup F$ are in $\mathcal{E}$. The pair $(X,\mathcal{E})$ is called a \emph{coarse space}. Let us recall that $E\circ F=\{(x,y)\mid (x,z)\in F,(z,y)\in E\,for\,some\,z\in X\}$ and $E^{-1}=\{(x,y)\mid (y,x)\in E\}$, for all $E,F\subseteq X\times X$. A member of $\mathcal{E}$ is called an \emph{entourage}. A coarse structure $\mathcal{E}$ is called \emph{unitary} if it contains the diagonal $\Delta =\{(x,x)\mid x\in X\}$. From now on by "coarse structure" we mean a "unitary coarse structure". A coarse structure is known as a large scale counterpart of a uniformity. In section 2 we try to introduce a large scale counterpart of proximity. For this reason we axiomatize the relation \emph{$A$ and $B$ are asymptotically alike} for two subsets $A$ and $B$ of a set $X$ and introduce the notion of \emph{asymptotic resemblance}. We call a set equipped with an asymptotic resemblance relation, an asymptotic resemblance (an AS.R.) space. In section 2 we show that how one can generalize basic concepts of coarse geometry (coarse maps, coarse connectedness, coarse subspace etc) by our definition. Also in this section we show that every coarse structures on a set $X$ can induce an asymptotic resemblance relation on $X$.\\
In section 3 we investigate the relation between coarse structures and asymptotic resemblance relations. We give an example of two different coarse structures on a set $X$ such that they induce a same asymptotic resemblance relation on $X$. We show how asymptotic resemblance relations on a set $X$ can admit an equivalence relation on the family of all coarse structures on $X$.\\
A coarse structure $\mathcal{E}$ on a topological space $X$ is said to be \emph{compatible with the topology} of $X$ if each entourage is contained in an open entourage. A compatible coarse structure on a topological space is called \emph{proper} if each bounded subset has compact closure. One can easily check that a unitary coarse structure is compatible with the topology of a space if and only if it contains an open entourage containing the diagonal \cite{Wr}. Let $\mathcal{E}$ be a proper coarse structure on a topological space $(X,\mathcal{T})$. A continuous and bounded map $f:X\rightarrow \mathbb{C}$ is called a \emph{Higson function} if for each $E\in \mathcal{E}$ and $\epsilon >0$ there exists a compact subset $K$ of $X$ such that $| f(x)-f(y)| <\epsilon$ for all $(x,y)\in E\setminus (K\times K)$. The family of all Higson functions is denoted by $C_{h}(X)$. The Gelfand-Naimark theorem on $C^{*}$-algebras shows that there is a compactification $hX$ of $X$, such that $C(hX)$ (the family of all continuous functions on $hX$) and $C_{h}(X)$ are isomorphic (section 2.3 of \cite{Roe}). The compactification $hX$ of $X$ is called the \emph{Higson compactification} of $X$. The compact set $\nu X=hX\setminus X$ is called the \emph{Higson corona} of $X$. In section 4 we use our notion of asymptotic resemblance to make a compactification of a space (the \emph{asymptotic compactification}) that in some cases agrees with the Higson compactification of coarse spaces. We are going to use the \emph{Wallman compactification} of a topological space to genarate our desired compactification. Let us recall the Wallman compactification of a topological space briefly (\cite{Wil}).\\
Let $(X,\mathcal{T})$ be a  Hausdorff topological space and let $\gamma X$ be the family of all closed ultrafilters on $X$. For each open subset $U$ of $X$, set $U^{*}=\{\mathcal{F}\in \gamma X\mid U\notin \mathcal{F}\}$. It is straightforward to show that $\mathcal{F}\in U^{*}$ if and only if $\mathcal{F}$ contains a subset of $U$. The family $\mathcal{B}=\{U^{*}\mid U\, is\, open\, in\, X\}$ is a basis for a topology on $\gamma X$ and $\gamma X$ is compact by this topology. Let $\sigma_{x}$ denotes the unique closed ultrafilter that converges to $x\in X$. The map $\sigma: X\rightarrow \gamma X$ defined by $\sigma(x)=\sigma_{x}$ is a topological embedding and $\gamma X$ is called the \emph{Wallman compactification} of $X$.\\
A \emph{cluster} $\mathcal{C}$ in a proximity space $(X,\delta)$ is a family of subsets of $X$ such that for all $A,B\in \mathcal{C}$ we have $A\delta B$, if $A,B\subseteq X$ and $A\bigcup B\in \mathcal{C}$ then $A\in \mathcal{C}$ or $B\in \mathcal{C}$ and if $A\delta B$ for all $B\in \mathcal{C}$ then $A\in \mathcal{C}$. A proximity space $(X,\delta)$ is said to be \emph{separated} if $x\delta y$ implies $x=y$, for all $x,y\in X$. A proximity $\delta$ on a topological space $(X,\mathcal{T})$ is said to be \emph{compatible with $\mathcal{T}$} if $a\in \bar{A}$ and $a\delta A$ are equivalent. Let $\mathfrak{X}$ denotes the family of all clusters in a separated proximity space $(X,\delta)$. For $\mathfrak{M} , \mathfrak{N} \subseteq \mathfrak{X}$ define $\mathfrak{M} \delta^{*} \mathfrak{N}$ if $A\subseteq X$ absorbs $\mathfrak{M}$ and $B\subseteq X$ absorbs $\mathfrak{N}$ then $A\delta B$. A set $D$ \emph{absorbs} $\mathfrak{M} \subseteq \mathfrak{X}$ means that $A\in \mathcal{C}$ for all $\mathcal{C} \in \mathfrak{M}$. The relation $\delta^{*}$ is a proximity on $\mathfrak{X}$. The pair $(\mathfrak{X},\delta^{*})$ is a compact proximity space and it is called the \emph{Smirnov compactification } of $(X,\delta)$ (section 7 of \cite{Nai}). In section 5 we introduce a proximity on an AS.R. space such that its Smirnov compactification agrees with the asymptotic compactification.\\

There are several equivalent definitions for \emph{asymptotic dimension} of a metric space (\cite{asdim}). In this paper by asymptotic dimension of a metric space $(X,d)$ we mean the following definition.
\begin{definition}
Let $X$ be a metric space. The inequality $\operatorname{asdim} X\leq n$ means that for each uniformly bounded cover $\mathcal{U}$ of $X$ there exists uniformly bounded cover $\mathcal{V}$ of $X$ such that $\mathcal{U}$ refines $\mathcal{V}$ and $\mu(\mathcal{V})\leq n+1$. For a family $\mathcal{M}$ of subsets of a set $X$, $\mu(\mathcal{M})$ denotes the multiplicity of $\mathcal{M}$ i.e the greatest number of elements of $\mathcal{M}$ that meets a point of $X$. By $\operatorname{asdim}X=n$ we mean that $\operatorname{asdim}X\leq n$ and $\operatorname{asdim}X\leq n-1$ does not hold. For a metric space $X$, $\operatorname{asdim}X$ is called the \emph{asymptotic dimension} of $X$.
\end{definition}
In section 6 we show how one can generalize the notion of asymptotic dimension to AS.R. spaces.\\

In this paper we denote by $d_{H}(A,B)$ the \emph{Hausdorff distance} between subsets $A$ and $B$ of a metric space $(X,d)$. Let us recall one more thing here. A proper map $f:X\rightarrow Y$ between metric spaces $(X,d)$ and $(Y,d^{\prime})$ is said to be a \emph{coarse map} if for each $r>0$ there exists $s>0$ such that $d(x,x^{\prime})<r$ implies $d^{\prime}(f(x),f(x^{\prime}))<s$.

\section{Asymptotic resemblance}

\begin{definition}\label{setare}
Let $X$ be a metric space. We say that two subsets $A$ and $B$ of $X$ are \emph{asymptotically alike} and we denote it by $A\lambda B$, if $d_{H}(A,B)<\infty$. We assume that $d_{H}(\emptyset,\emptyset)=0$ and $d_{H}(\emptyset,A)=\infty$ for all $\emptyset \neq A\subseteq X$.
\end{definition}
 Let us denote the open ball of radius $r>0$ around $x\in X$ by $\textbf{B}(x,r)$ and let $\textbf{B}(A,r)=\bigcup_{a\in A}\textbf{B}(x,r)$ for each subset $A$ of $X$. The above definition states that $A\lambda B$ if and only if there is $r>0$ such that $A\subseteq \textbf{B}(B,r)$ and $B\subseteq \textbf{B}(A,r)$.\\
 Let $(x_{n})_{n\in \mathbb{N}}$ and $(y_{n})_{n\in \mathbb{N}}$ be two sequences in a metric space $(X,d)$. If there exists $k>0$ such that $d(x_{n},y_{n})<k$ for all $n\in \mathbb{N}$ then we have $\{x_{i}\mid i\in I\} \lambda \{y_{i}\mid i\in I\}$ for each $I\subseteq \mathbb{N}$. The converse is also true.
\begin{lemma}\label{hhhh}
Let $(X,d)$ be a metric space. Suppose that $(x_{n})_{n\in \mathbb{N}}$ and $(y_{n})_{n\in \mathbb{N}}$ are two sequences in $X$ such that for each subset $I$ of $\mathbb{N}$ we have $\{x_{i}\mid i\in I\}\lambda \{y_{i}\mid i\in I\}$. Then, there exists $k>0$ such that $d(x_{n},y_{n})<k$ for all $n\in \mathbb{N}$.
\end{lemma}
\begin{proof}
Suppose, contrary to our claim, that for each $n\in \mathbb{N}$ there is some $i_{n}\in \mathbb{N}$ such that $d(x_{i_{n}},y_{i_{n}})>n$. Without loss of generality we can assume that we have $i_{n}=n$ for each $n\in \mathbb{N}$. We derive a contradiction by two steps.\\
Step 1: We claim that for each $x\in X$ and $s>0$, the index set $I=\{i\in \mathbb{N}\mid x_{i}\in \textbf{B}(x,s)\}$ is finite. Let $C=\{x_{i}\mid i\in I\}$ and $D=\{y_{i}\mid i\in I\}$. Then $C\lambda D$, let $d_{H}(C,D)=r$. If $j\in I$ we have:
$$i<d(x_{i},y_{i})\leq d(x_{i},x_{j})+d(x_{j},y_{i})<2s+r$$
This implies that $I$ is finite. Similarly we can prove that for each bounded subset $D$ of $X$ the index set $J=\{j\in \mathbb{N}\mid y_{j}\in D\}$ is finite.\\
Step 2: Set $E_{k}=\{x_{n}\mid n\geq k\}\bigcup \{y_{n}\mid n\geq k\}$ for $k\in \mathbb{N}$. By step 1, for each bounded set $D$, there is $k\in \mathbb{N}$ such that $E_{k}\bigcap D=\emptyset$. Let $k_{1}=1$. For each $i\in \mathbb{N}$ choose $k_{i+1}\in \mathbb{N}$ such that $E_{k_{i+1}}\bigcap
\textbf{B}(\{x_{k_{i}},y_{k_{i}}\},k_{i})=\emptyset$. Let $A=\{x_{k_{i}}\mid i\in N\}$ and $B=\{y_{k_{i}}\mid i\in N\}$. We have $A\lambda B$ so there exists $s>0$ such that $A\subseteq \textbf{B}(B,s)$ and $B\subseteq \textbf{B}(A,s)$. Now choose $k_{i}>s$. For $j,l\geq i$ let $\alpha=\min\{j,l\}$, we have $d(y_{k_{j}},x_{k_{l}})\geq k_{\alpha}\geq k_{i}>s$. Therefore for each $j\geq i$ we have $y_{k_{j}}\in\textbf{ B}(x_{k_{r}},s)$ for some $r=1,...,i-1$, which means the set $\{j\in \mathbb{N}\mid y_{j}\in \bigcup_{r=1}^{i-1} \textbf{B}(x_{k_{r}},s)\}$ is infinite and it contradicts step 1 of the proof.

\end{proof}

It is well known that a map $f:X\rightarrow Y$ between metric spaces $X$ and $Y$ is uniformly continuous if and only if for two subsets $A$ and $B$ of $X$, $A\delta B$ implies $f(A)\delta f(B)$ (\cite{Nai} 4.8). Where $\delta$ denotes the metric proximity i.e $A\delta B$ if and only if $d(A,B)=0$. The following theorem is the large scale counterpart of this fact.
\begin{theorem}\label{5setare}
Let $X$ and $Y$ be two metric spaces. A proper map $f:X\rightarrow Y$ is a coarse map if and only if for each asymptotically alike subsets $A$ and $B$ of $X$, $f(A)$ and $f(B)$ are asymptotically alike too.
\end{theorem}
\begin{proof}
Suppose that $f:X\rightarrow Y$ is a coarse map. Let $A$ and $B$ be two subsets of $X$ such that $A\subseteq \textbf{B}(B,r)$ and $B\subseteq \textbf{B}(A,r)$ for some $r>0$. By hypothesis there exists $s>0$ such that $d(x,x^{\prime})<r$ yields $d(f(x),f(x^{\prime}))<s$, so $f(A)\subseteq \textbf{B}(f(B),s)$ and $f(B)\subseteq \textbf{B}(f(A),s)$.\\
To prove the converse, assume that $f$ is not a coarse map. So there are $r>0$ and sequences $x_{n}$ and $y_{n}$ in $X$ such that $d(x_{n},y_{n})<r$ and $d(f(x_{n}),f(y_{n}))>n$. But the sequences $(f(x_{n}))_{n\in \mathbb{N}}$ and $(f(y_{n}))_{n\in \mathbb{N}}$ satisfy the hypothesis of \ref{hhhh}, a contradiction.
\end{proof}
\begin{proposition}\label{2setare}
Let $X$ be a metric space. The relation $\lambda$ defined in \ref{setare} is an equivalence relation on the family of all subsets of $X$ and it has following properties:\\
i) $A_{1}\lambda B_{1}$ and $A_{2}\lambda B_{2}$ implies $(A_{1}\bigcup A_{2})\lambda (B_{1}\bigcup B_{2})$.\\
ii) $(B_{1}\bigcup B_{2})\lambda A$ and $B_{1},B_{2}\neq \emptyset$ implies that there are nonempty subsets $A_{1}$ and $A_{2}$ of $A$ such that $A=A_{1}\bigcup A_{2}$ and we have $B_{i}\lambda A_{i}$ for $i\in \{1,2\}$.
\end{proposition}
\begin{proof}
It is straightforward to show $\lambda$ is an equivalence relation on the family of all subsets of $X$ and it satisfies property (i). For property (ii) assume that $B_{1}\bigcup B_{2}\subseteq \textbf{B}(A,r)$ and $A\subseteq \textbf{B}(B_{1}\bigcup B_{2},r)$ for some $r>0$ and $B_{1},B_{2}\neq \emptyset$. For $i\in \{1,2\}$ let $A_{i}=\textbf{B}(B_{i},r)\bigcap A$. We have $A=A_{1}\bigcup A_{2}$ and $A_{i}\lambda B_{i}$ for $i\in \{1,2\}$.
\end{proof}
\begin{definition}\label{4setare}
Let $X$ be a set. We call a binary relation $\lambda$ on the power set of $X$ an \emph{asymptotic resemblance} (an AS.R.) if it is an equivalence relation on the family of all subsets of $X$ and satisfies the properties (i) and (ii) of \ref{2setare}. For subsets $A$ and $B$ of $X$ we say that $A$ and $B$ are \emph{asymptotically alike} if $A\lambda B$. By $A\bar{\lambda} B$ we mean that $A$ and $B$ are not asymptotically alike. We call the pair $(X,\lambda)$ an AS.R. space.
\end{definition}
In a metric space $(X,d)$, we call the relation defined in \ref{setare} \emph{the AS.R. associated to the metric $d$} on $X$.
\begin{proposition}\label{6setare}
Let $(X,\lambda)$ be an AS.R. space. If $A\lambda B$ and $\emptyset \neq A_{1}\subseteq A$ then there is $\emptyset \neq B_{1}\subseteq B$ such that $A_{1}\lambda B_{1}$.
\end{proposition}
\begin{proof}
It is an immediate consequence of \ref{2setare} (ii).
\end{proof}
\begin{proposition}\label{sss}
Let $\lambda$ be an AS.R. on a set $X$. Suppose that $A,B,C\subseteq X$ and $A\subseteq B\subseteq C$. If $A\lambda C$ then $A\lambda B$.
\end{proposition}
\begin{proof}
 The property (i) of \ref{2setare} leads to $((B\setminus A)\bigcup A)\lambda ((B\setminus A)\bigcup C)$. Thus $B\lambda C$ and since $\lambda$ is an equivalence relation $A\lambda B$.
\end{proof}
Let us recall that on a coarse space $(X,\mathcal{E})$, $E(A)=\{y\in X\mid (x,y)\in E\,for\,some\,x\in A\}$ for all $E\in \mathcal{E}$ and all $A\subseteq X$.
\begin{example}\label{coarse}
Suppose that $\mathcal{E}$ is a coarse structure on a set $X$. For any two subsets $A$ and $B$ of $X$, define $A\lambda_{\mathcal{E}} B$ if $A\subseteq E(B)$ and $B\subseteq E(A)$ for some $E\in \mathcal{E}$. The relation $\lambda_{\mathcal{E}}$ is an asymptotic resemblance on $X$. We call $\lambda_{\mathcal{E}}$ \emph{the AS.R. associated to the coarse structure} $\mathcal{E}$ on $X$.
\end{example}
In the next section we will investigate the relation between coarse structures and asymptotic resemblance relations in more details.
\begin{example}
Let $X$ be a set. For any two subsets $A$ and $B$ of $X$, define $A\lambda B$ if $A\Delta B=(A\setminus B)\bigcup (B\setminus A)$ is finite. The relation $\lambda$ is an AS.R. on $X$ that we call it the \emph{discrete asymptotic resemblance} on a set $X$.
\end{example}
\begin{definition}
Let $\lambda$ be an AS.R. on a set $X$. We say a subset $A$ of $X$ is \emph{bounded} if $A\lambda x$, for some $x\in X$. We assume that the empty set is bounded.
\end{definition}
Let $\lambda$ be the AS.R. associated to a coarse structure $\mathcal{E}$ on a set $X$. It is easy to verify that $D\subseteq X$ is bounded if and only if it is bounded with respect to $\mathcal{E}$
\begin{proposition}\label{4setare}
Let $\lambda$ be an AS.R. on a set $X$ and let $A\subseteq X$. If $A\lambda x$ for some $x\in X$ and $\emptyset \neq B\subseteq A$ then $B\lambda x$. Thus all subsets of a bounded set are bounded.
\end{proposition}
\begin{proof}
It is an immediate consequence of \ref{6setare}.
\end{proof}
\begin{example}
Suppose that $G$ is a group. For two subsets $A$ and $B$ of $G$ define $A\lambda_{l}B$ if there exists a finite subset $K$ of $G$ such that $A\subset BK$ and $B\subseteq AK$. We call $\lambda_{l}$ the \emph{left AS.R.} on $G$. Similarly one can define the \emph{right AS.R.} on $G$. In both cases a subset $D$ of $G$ is bounded if and only if it is finite. If $G$ is an Abelian group then $\lambda_{r}$ and $\lambda_{l}$ obviously coincide. However they are different in general case (\cite{Alt}).
\end{example}
\begin{example}\label{lll}
Suppose that $A$ and $B$ are two subsets of the real line $\mathbb{R}$. Define $A\lambda B$ if there exists $r>0$ such that $A\subseteq \bigcup_{b\in B}(b-r,+\infty)$ and $B\subseteq \bigcup_{a\in A}(a-r,+\infty)$. It is straightforward to show that $\lambda$ is an equivalence relation on the family of all subsets of $\mathbb{R}$ and it satisfies (i) of \ref{2setare}. Now suppose that $A\lambda (B_{1}\bigcup B_{2})$ and $B_{1},B_{2}\neq \emptyset$. So there is $r>0$ such that we have $A\subseteq \bigcup_{b\in B_{1}\bigcup B_{2}}(b-r,+\infty)$ and $B_{1}\bigcup B_{2}\subseteq \bigcup_{a\in A}(a-r,+\infty)$. Let $A_{1}^{\prime}=(\bigcup_{b\in B_{1}}(b-r,+\infty))\bigcap A$. If $B_{1}\subseteq \bigcup_{a\in A_{1}^{\prime}}(a-r,\infty)$ so $A_{1}^{\prime}\lambda B_{1}$ and we can let $A_{1}=A_{1}^{\prime}$. Now assume that there is $b_{1}\in B_{1}$ such that $b_{1}\leq a-r$ for all $a\in A_{1}^{\prime}$. Since $A\lambda (B_{1}\bigcup B_{2})$ there is $a_{1}\in A$ such that $a_{1}-r<b_{1}$. Let $A_{1}=A_{1}^{\prime}\bigcup \{a_{1}\}$ and $r_{1}=\max\{\mid b_{1}-a_{1}\mid +1 ,r\}$. Since $a_{1}$ is not in $A_{1}^{\prime}$, $a_{1}\leq b-r<b$ for all $b\in B_{1}$. Thus $B_{1}\subseteq (a_{1}-r_{1},+\infty)$ which leads to $A_{1}\lambda B_{1}$. Similarly one can find $A_{2}\subseteq A$ such that $A_{2}\lambda B_{2}$ and $A=A_{1}\bigcup A_{2}$. Let $\lambda_{d}$ denotes the AS.R. associated to the standard metric on $\mathbb{R}$. It is easy to show that $A\lambda_{d} B$ yields $A\lambda B$, for all $A,B\subseteq \mathbb{R}$. A set $D\subseteq \mathbb{R}$ is bounded with respect to $\lambda$ if and only if $D\subseteq (a,+\infty)$ for some $a\in \mathbb{R}$. There is not any metric on $\mathbb{R}$ such that we have $A\lambda B$ if and only if $d_{H}(A,B)<\infty$ for all subsets $A$ and $B$ of $\mathbb{R}$. Suppose the contrary. For each $n\in \mathbb{N}$ the interval $(-\infty,-n)$ is unbounded. We choose $b_{n}<-n$ such that $d(-n,b_{n})>n$. But the sequences $(-n)_{n\in \mathbb{N}}$ and $(b_{n})_{n\in \mathbb{N}}$ satisfy the hypothesis of \ref{hhhh}, a contradiction.
\end{example}
\begin{definition}
Let $(X,\lambda_{1})$ and $(Y,\lambda_{2})$ be two AS.R. spaces. We call a map $f:X\rightarrow Y$ an \emph{AS.R. mapping} if\\
i) $f^{-1}(B)$ is bounded in $X$ for each bounded subset $B$ of $Y$. (Properness)\\
ii) $A\lambda_{1}B$ implies $f(A)\lambda_{2}f(B)$, for all subsets $A$ and $B$ of $X$.
\end{definition}
In fact \ref{5setare} says that for metric spaces $X$ and $Y$ a map $f:X\rightarrow Y$ is a coarse map if and only if it is an AS.R. mapping for the AS.R.s associated to their metrics.
\begin{definition}
Let $(Y,\lambda)$ be an AS.R. space and let $X$ be a set. We say that two maps $f:X\rightarrow Y$ and $g:X\rightarrow Y$ are \emph{close} if we have $f(A)\lambda g(A)$ for each subset $A$ of $X$.
\end{definition}
\begin{proposition}
Let $(Y,d)$ be a metric space and let $\lambda$ be the AS.R. associated to $d$. Two maps $f:X\rightarrow Y$ and $g:X\rightarrow Y$ are close if and only if there is some $k>0$ such that $d(f(x),g(x))<k$ for all $x\in X$.
\end{proposition}
\begin{proof}
The proof of "only if" part is straightforward. Now suppose that $f$ and $g$ are close maps. Assume that on the contrary, for all $n\in \mathbb{N}$ there exists $x_{n}\in X$ such that $d(f(x_{n}),g(x_{n}))>n$. But the sequences $(f(x_{n}))_{n\in \mathbb{N}}$ and $(g(x_{n}))_{n\in \mathbb{N}}$ satisfy the hypothesis of \ref{hhhh}, a contradiction.
\end{proof}
\begin{definition}
Let $(X,\lambda_{1})$ and $(Y,\lambda_{2})$ be two AS.R. spaces. We call an AS.R. mapping $f:X\rightarrow Y$ an \emph{asymptotic equivalence} if there exists an AS.R. mapping $g:Y\rightarrow X$ such that $gof$ and $fog$ are close to the identity maps $i_{X}:X\rightarrow X$ and $i_{Y}:Y\rightarrow Y$ respectively. We say AS.R. spaces $(X,\lambda_{1})$ and $(Y,\lambda_{2})$ are \emph{asymptotically equivalent} if there exists an asymptotic equivalence $f:X\rightarrow Y$.
\end{definition}
\begin{proposition}
Let $(X,\lambda_{1})$ and $(Y,\lambda_{2})$ be two AS.R. spaces. Suppose that $f:X\rightarrow Y$ and $g:X\rightarrow Y$ are two close maps. If $f$ is an AS.R. mapping then so is $g$ and if $f$ is an asymptotic equivalence then so is $g$.
\end{proposition}
\begin{proof}
We are going to prove that if $f$ is a proper map then so is $g$. Other parts of the corollary are straightforward results of the property that $\lambda_{1}$ and $\lambda_{2}$ are equivalence relations on the family of all subsets of $X$ and $Y$.\\
Let $D\subseteq Y$ be a bounded set. We have $f(g^{-1}(D))\lambda_{2}g(g^{-1}(D))$ so $f(g^{-1}(D))$ is bounded. Thus $f^{-1}(f(g^{-1}(D)))$ is bounded and \ref{4setare} leads to $g^{-1}(D)$ is bounded.
\end{proof}
\begin{definition}
Let $(X,\lambda)$ be an AS.R. space and let $Y$ be a nonempty subset of $X$. For all two subsets $A$ and $B$ of $Y$, define $A\lambda_{Y}B$ if $A\lambda B$. The pair $(Y,\lambda_{Y})$ is an AS.R. space and we call $\lambda_{Y}$ the \emph{subspace} AS.R. induced by $\lambda$ on $Y$.
\end{definition}
\begin{lemma}
Let $(X,\lambda)$ and $(Y,\lambda^{\prime})$ be two AS.R. spaces. Suppose that $f:X\rightarrow Y$ is an asymptotic equivalence and $\emptyset \neq C\subseteq X$. Then $f\mid_{C}:(C,\lambda_{C})\rightarrow (f(C),\lambda^{\prime}_{f(C)})$ is an asymptotic equivalence too.
\end{lemma}
\begin{proof}
Let $g:Y\rightarrow X$ be an AS.R. mapping such that $g\circ f$ and $f\circ g$  are close maps to identity map $i_{X}:X\rightarrow X$ and $i_{Y}:Y\rightarrow Y$ respectively. Let $q:f(C)\rightarrow C$ be a map such that $f\circ q(a)=a$ for each $a\in f(C)$. Suppose that $D\subseteq C$ is bounded. Since $g\circ f(D)\lambda D$, $g\circ f(D)$ is a bounded subset of $X$. We have $q^{-1}(D)\subseteq f(D)\subseteq g^{-1}(g\circ f(D))$, \ref{4setare} shows $q^{-1}(D)$ is bounded. Assume that $A,B\subseteq f(C)$ and $A\lambda^{\prime}_{f(C)}B$. We have $g\circ f (q(A))\lambda q(A)$ and since $f(q(A))=A$, $q(A)\lambda g(A)$. Similarly $q(B)\lambda g(B)$ and it leads to $q(A)\lambda_{C} q(B)$. Therefore $q$ is an AS.R. mapping. Now let $A\subseteq C$. We have $f(q\circ f(A))=f(A)$ so $g(f(q\circ f(A)))=gof(A)\lambda A$. Also we have $g\circ f(q\circ f(A))\lambda q\circ f(A)$ and it leads to $q\circ f(A)\lambda_{C} A$. Therefore $f\mid_{C}:C\rightarrow f(C)$ is an asymptotic equivalence.
\end{proof}
\begin{definition}
We call an AS.R. space $(X,\lambda)$ \emph{asymptotically connected} if we have $x\lambda y$ for all $x,y\in X$.
\end{definition}
It is immediate that the AS.R. associated to a connected coarse structure is asymptotically connected.
\begin{proposition}
An AS.R. space $(X,\lambda)$ is asymptotically connected if and only if for each nonempty subsets $A$ and $B$ of $X$, $A\Delta B$ is finite yields $A\lambda B$.
\end{proposition}
\begin{proof}
The "if" part is trivial. Assume that $A\setminus B=\{x_{1},...,x_{n}\}$ and $B\setminus A=\{y_{1},...,y_{m}\}$. By using (i) of \ref{2setare} and asymptotically connectedness of $\lambda$ we can conclude $(A\setminus B)\lambda (B\setminus A)$. By (i) of \ref{2setare} we have $((A\setminus B)\bigcup (A\bigcap B))\lambda ((B\setminus A)\bigcup (A\bigcap B))$. Thus $A\lambda B$.
\end{proof}
\section{Coarse structures and asymptotic resemblance relations}
In \ref{coarse} we stated that every coarse structure $\mathcal{E}$ on a set $X$ induces an AS.R. on $X$. We denoted this AS.R. by $\lambda_{\mathcal{E}}$. The following example shows that two different coarse structures may induce a same AS.R. relation.
\begin{example}
Let $X=\mathbb{N}$. Assume that $\mathcal{E}_{1}$ and $\mathcal{E}_{2}$ denote two families of subsets of $X\times X$ such that:\\
i) $E\in \mathcal{E}_{1}$ if and only if $E(A)$  and $E^{-1}(A)$ are finite for all finite $A\subseteq \mathbb{N}$.\\
ii) $E\in \mathcal{E}_{2}$ if and only if there exists $n_{E}\in \mathbb{N}$ such that $E(x)$ and $E^{-1}(x)$ have at most $n_{E}$ members, for all $x\in X$.\\
Both families $\mathcal{E}_{1}$ and $\mathcal{E}_{2}$ are coarse structures on $X$ (examples 2.8 and 2.44 of \cite{Roe}). It is immediate that $\mathcal{E}_{2}$ is a proper subset of $\mathcal{E}_{1}$. For two subsets $A,B$ of $X$ we claim that $A\lambda_{\mathcal{E}_{2}}B$ if and only if $A$ and $B$ are both finite or $A$ and $B$ are both infinite. It is straightforward to show that if $A\lambda_{\mathcal{E}_{2}}B$ and $A$ is finite then so is $B$ and if $A$ and $B$ are both finite then $A\lambda_{\mathcal{E}_{2}}B$. Suppose that $A$ and $B$ are both infinite. Let $A=\{a_{n}\mid n\in \mathbb{N}\}$ and $B=\{b_{n}\mid n\in \mathbb{N}\}$ and assume that $a_{n}<a_{n+1}$ and $b_{n}<b_{n+1}$ for all $n\in \mathbb{N}$. Let $E=\{(a_{n},b_{n})\mid n\in \mathbb{N}\}\bigcup \{(b_{n},a_{n})\mid n\in \mathbb{N}\}$. Clearly $E\in \mathcal{E}_{2}$ and $n_{E}=2$. We have $A\subseteq E(B)$ and $B\subseteq E(A)$ so $A\lambda_{\mathcal{E}_{2}}B$. Since $\mathcal{E}_{2}\subseteq \mathcal{E}_{1}$ one can easily shows that $A\lambda_{\mathcal{E}_{1}}B$ if and only if $A$ and $B$ are both finite or $A$ and $B$ are both infinite. Thus $\lambda_{\mathcal{E}_{1}}=\lambda_{\mathcal{E}_{2}}$.
\end{example}
Let $\lambda$ be an AS.R. on a set $X$. We denote the family of all coarse structures that induce $\lambda$ by $\mathcal{E}(\lambda)$. Let us recall that for two coarse structures $\mathcal{E}_{1}$ and $\mathcal{E}_{2}$ on a set $X$, $\mathcal{E}_{2}$ is called to be \emph{coarser} than $\mathcal{E}_{1}$ if $\mathcal{E}_{1}\subseteq \mathcal{E}_{2}$ (section 2.1 of \cite{Roe}).

\begin{proposition}
Let $\lambda$ be an AS.R. on a set $X$. If $\mathcal{E}(\lambda)\neq \emptyset$ then there is a coarse structure $\mathcal{E}_{\lambda}\in \mathcal{E}(\lambda)$ such that $\mathcal{E}_{\lambda}$ is coarser than each member of $\mathcal{E}(\lambda)$.
\end{proposition}
\begin{proof}
Let $\mathcal{E}_{\lambda}$ be the family of all $E\subseteq X\times X$ such that $\pi_{1}(F)\lambda \pi_{2}(F)$ for all $F\subseteq E$, where $\pi_{1}$ and $\pi_{2}$ denote projection maps onto first and second factors, respectively. Since $\lambda$ is an equivalence relation $\Delta \in \mathcal{E}_{\lambda}$ and $E^{-1}\in \mathcal{E}_{\lambda}$ for all $E\in \mathcal{E}_{\lambda}$. By property i) of \ref{2setare} it one can easily shows that $E\bigcup F\in \mathcal{E}_{\lambda}$ for all $E,F\in \mathcal{E}_{\lambda}$. Let $E,F\in \mathcal{E}_{\lambda}$ and suppose that $H\subseteq E\circ F$. Set $O_{1}=\{(x,y)\in X\times X\mid x\in \pi_{1}(H), y\in F(x)\bigcap \pi_{1}(E)\}$ and $O_{2}=\{(x,y)\in X\times X\mid x\in \pi_{2}(O_{1}), y\in E(x)\}$. We have $O_{1}\subseteq F$ and $O_{2}\subseteq E$ so $\pi_{1}(H)=\pi_{O_{1}}\lambda \pi_{2}(O_{1})$ and $\pi_{2}(O_{1})=\pi_{1}(O_{2})\lambda \pi_{2}(O_{2})=\pi_{2}(H)$. So $\pi_{1}(H)\lambda \pi_{2}(H)$ which leads to $E\circ F\in \mathcal{E}_{\lambda}$. Therefore $\mathcal{E}_{\lambda}$ is a coarse structure on $X$. Suppose that $\mathcal{E}\in \mathcal{E}(\lambda)$. It is straightforward by the definition to show that if $E\in \mathcal{E}$ and $F\subseteq E$ then $\pi_{1}(F)\lambda\pi_{2}(F)$. So $\mathcal{E}\subseteq \mathcal{E}_{\lambda}$. Thus $\mathcal{E}_{\lambda}$ is coarser than each member of $\mathcal{E}(\lambda)$. It remains to show $\mathcal{E}_{\lambda}\in \mathcal{E}(\lambda)$. Suppose that $A,B\subseteq X$ and $A\subseteq E(B)$ and $B\subseteq E(A)$, for some $E\in \mathcal{E}_{\lambda}$. Let $F_{1}=\{(a,b)\in E\mid a\in A, b\in B\}$ and $F_{2}=\{(b,a)\in E\mid a\in A, b\in B\}$. So $A=\pi_{1}(F_{1})\lambda \pi_{2}(F_{1})$ and $B=\pi_{1}(F_{2})\lambda \pi_{2}(F_{2})$. We have $A=(\pi_{2}(F_{2})\bigcup (A\setminus \pi_{2}(F_{2})))\lambda (B\bigcup A\setminus \pi_{2}(F_{2}))$. Sine $A$ is asymptotically alike to $\pi_{2}(F_{1})\subseteq B$, there is a subset $L$ of $B$ such that $(A\setminus \pi_{2}(F_{2}))\lambda L$, by \ref{6setare}. Therefore $A\lambda (B\bigcup L)=B$. Since $\mathcal{E}(\lambda)\neq \emptyset$ and $\mathcal{E}_{\lambda}$ is greater than each member of $\mathcal{E}(\lambda)$ it is straightforward to show that $A\lambda B$ implies there is $E\in \mathcal{E}_{\lambda}$ such that $A\subseteq E(B)$ and $B\subseteq E(A)$, for all $A,B\subseteq X$.
\end{proof}
In fact asymptotic resemblance relations on a set $X$ defines an equivalence relation on the family of all coarse structures on $X$. Two coarse structures on $X$ are equivalent if they induce the same asymptotic resemblance relation. The previous proposition shows that these equivalence classes have a biggest member. One can compare this with similar arguments about the relation between uniform structures and proximity in section 12 of \cite{Nai}.

\section{Asymptotic compactification}
\begin{definition}
Let $(X,\mathcal{T})$ be a topological space and let $\lambda$ be an AS.R. on $X$. We say that an open subset $U$ of $X$ is an \emph{asymptotic neighbourhood} of $A\subseteq X$ if $A\subseteq U$ and $A\lambda U$. We call $\lambda$ a \emph{compatible AS.R. with $\mathcal{T}$} if\\
i) Each subset of $X$ has an asymptotic neighbourhood.\\
ii) $A \lambda \bar{A}$ for all $A\subseteq X$.
\end{definition}

\begin{proposition}
Let $(X,\mathcal{T})$ be a topological space and let $\mathcal{E}$ be a coarse structure compatible with $\mathcal{T}$. Then the AS.R. associated to $\mathcal{E}$ is compatible with $\mathcal{T}$ too.
\end{proposition}
\begin{proof}
Assume that $E$ is a symmetric open entourage containing the diagonal. For $A\subseteq X$, $E(A)$ is an asymptotic neighbourhood of $A$. Let $a\in \bar{A}$. $E(a)$ is an open neighbourhood of $a$, so $E(a)\bigcap A\neq \emptyset$. Let $a^{\prime}\in E(a)\bigcap A$. Since $(a,a^{\prime})\in E^{-1}=E$, $a\in E(a^{\prime})\subseteq E(A)$. Thus $\bar{A}\subseteq E(A)$ and this leads to $A\lambda \bar{A}$.
\end{proof}
\begin{definition}
We call two subsets $A_{1}$ and $A_{2}$ of an AS.R. space $(X,\lambda)$ \emph{asymptotically disjoint} if for all unbounded subsets $L_{1}\subseteq A_{1}$ and $L_{2}\subseteq A_{2}$ we have $L_{1}\bar{\lambda} L_{2}$. We say that an AS.R. space $(X,\lambda)$ is \emph{asymptotically normal} if for asymptotically disjoint subsets $A_{1}$ and $A_{2}$ of $X$, there exist $X_{1}\subseteq X$ and $X_{2}\subseteq X$ such that $X=X_{1}\bigcup X_{2}$ and $A_{i}$ and $X_{i}$ are asymptotically disjoint for $i\in \{1,2\}$.
\end{definition}
Let $B$ a bounded subset of an AS.R. space $(X,\lambda)$. Then $B$ is asymptotically disjoint from all $A\subseteq X$.\\
In \cite{Ind} two subsets $A$ and $B$ of a metric space $(X,d)$ are called asymptotically disjoint if for some $x_{0}\in X$, $\lim_{r\rightarrow \infty} d(A\setminus\textbf{B}(x_{0},r),B\setminus\textbf{B}(x_{0},r))=\infty$. The following proposition shows that this definition is equivalent to our definition of asymptotical disjointness on metric spaces.
\begin{proposition}
Let $(X,d)$ be a metric space and let $\lambda$ be the associated AS.R. to $d$. Two unbounded subsets $A$ and $B$ of $X$ are asymptotically disjoint if and only if for some $x_0$ in $X$, $\lim_{r\rightarrow \infty}d(A\setminus\textbf{B}(x_{0},r),B\setminus\textbf{B}(x_{0},r))=\infty$.
\end{proposition}
\begin{proof}
Let $x_{0}\in X$ be a fixed point. Suppose that $A$ and $B$ are two asymptotically disjoint subsets of $X$. Assume that on the contrary, $\lim_{r\rightarrow \infty}d(A\setminus\textbf{B}(x_{0},r),B\setminus\textbf{B}(x_{0},r))\neq \infty$. Thus there exists $N\in \mathbb{N}$ such that for each $m\in \mathbb{N}$ we have $d(A\setminus\textbf{B}(x_{0},r_{m}),B\setminus\textbf{B}(x_{0},r_{m}))<N$ for some $r_{m}\geq m$. We choose $x_{m}\in A\setminus \textbf{B}(x_{0},r_{m})$ and $y_{m}\in B\setminus \textbf{B}(x_{0},r_{m})$ such that $d(x_{m},y_{m})<N$. Let $L_{1}=\{x_{m}\mid m\in \mathbb{N}\}$ and $L_{2}=\{y_{m}\mid m\in \mathbb{N}\}$. Thus $L_{1}\subseteq A$ and $L_{2}\subseteq B$ are two unbounded subsets and $L_{1}\lambda L_{2}$, a contradiction.\\
To prove the converse, let $A,B\subseteq X$ and suppose that $\lim_{r\rightarrow \infty}d(A\setminus\textbf{B}(x_{0},r),B\setminus\textbf{B}(x_{0},r))=\infty$. Assume that on the contrary, there are unbounded subsets $L_{1}\subseteq A$ and $L_{2}\subseteq B$ such that $d_{H}(L_{1},L_{2})<N$, for some $N\in \mathbb{N}$. Since $L_{1}$ is unbounded for each $n\in \mathbb{N}$, there exists $x_{n}\in L_{1}\setminus \textbf{B}(x_{0},n)$ such that $d(x_{n},b)>N$ for all $b\in \textbf{B}(x_{0},n)\bigcap L_{2}$. Thus there is $y_{n}\in L_{2}\setminus \textbf{B}(x_{0},n)$ such that $d(x_{n},y_{n})<N$. Then $x_{n}\in A\setminus \textbf{B}(x_{0},n)$ and $y_{n}\in B\setminus \textbf{B}(x_{0},n)$ for all $n\in \mathbb{N}$. Thus $\lim_{r\rightarrow \infty}d(A\setminus\textbf{B}(x_{0},r),B\setminus\textbf{B}(x_{0},r))\neq \infty$, a contradiction.
\end{proof}
\begin{proposition}
Let $(X,d)$ be a metric space and let $\lambda$ be the AS.R. associated to $d$. Then $(X,\lambda)$ is an asymptotically normal AS.R. space.
\end{proposition}
\begin{proof}
Assume that $A$ and $B$ are asymptotically disjoint subsets of $X$. For $i\in \mathbb{N}\bigcup \{0\}$, let $A_{i}=\{x\mid d(x,A)\leq i+1\}\bigcap \{x\mid d(x,B)\geq i\}$ and $B_{i}=\{x\mid d(x,B)\leq i+1\}\bigcap \{x\mid d(x,A)\geq i\}$. Suppose that $X_{1}=\bigcup_{i=0}^{\infty}B_{i}$ and $X_{2}=\bigcup_{i=0}^{\infty}A_{i}$. For $x\in X$ assume that $i\leq d(x,A)\leq i+1$ and $j\leq d(x,B)\leq j+1$. If $i=j$ then $x\in A_{i}=B_{j}$. If $i<j$ then $i+1\leq j$ so $x\in A_{i}$. Thus $X=X_{1}\bigcup X_{2}$. We claim that $A$ and $X_{1}$ are asymptotically disjoint. Suppose that, on the contrary to our claim, there are unbounded subsets $L_{1}\subseteq A$ and $L_{2}\subseteq X_{1}$ such that $L_{1}\lambda L_{2}$ i.e. $L_{1}\subseteq \textbf{B}(L_{2},n)$ and $L_{2}\subseteq \textbf{B}(L_{1},n)$ for some $n\in \mathbb{N}$. Thus $L_{2}\subseteq \textbf{B}(A,n)$ and it leads to $L_{2}\subseteq \bigcup_{i=0}^{n-1}B_{i}$. So $L_{2}\subseteq \textbf{B}(B,n)$. Let $L_{3}=\textbf{B}(L_{2},n)\bigcap B$. We have $L_{3}\lambda L_{2}$ and it leads to $L_{3}\lambda L_{1}$, a contradiction. Therefore $A$ and $X_{1}$ are asymptotically disjoint. Similarly one can show that $B$ and $X_{2}$ are asymptotically disjoint.
\end{proof}
Let $X$ be a Hausdorff and locally compact topological space and let $\alpha X$ be a compactification of $X$. Let us recall that the topological coarse structure on $X$ associated to $\alpha X$ is the family of all $E\subseteq X\times X$ such that the closure of $E$ meets $(\alpha X\times \alpha X)\setminus (X\times X)$only in the diagonal (definition 2.28 of \cite{Roe}). It is known that topological coarse structures associated to a second countable compactifications are not metrizable (example 2.53 of \cite{Roe}). The following proposition shows that the class of all asymptotic normal AS.R. spaces is much bigger than the family of all metric spaces.
\begin{proposition}
Let $X$ be a Hausdorff and locally compact metric space and $\alpha X$ be a first countable compactification of $X$. Let $\mathcal{E}$ be the topological coarse structure associated to $\alpha X$ and $\lambda$ be the AS.R. associated to $\mathcal{E}$. Then $\lambda$ is asymptotically normal.
\end{proposition}
\begin{proof}
First we prove that $A$ and $B$ are asymptotically disjoint subsets of $X$ if and only if $\bar{A}\bigcap \bar{B}\bigcap (\alpha X\setminus X)=\emptyset$. Let $\omega \in \bar{A}\bigcap \bar{B}\bigcap (\alpha X\setminus X)$ for $A,B\subseteq X$. There are sequences $(x_{n})_{n\in \mathbb{N}}$ and $(y_{n})_{n\in \mathbb{N}}$ in $A$ and $B$ respectively such that they converge to $\omega$. Let $E=\{(x_{n},y_{n})\mid n\in \mathbb{N}\}$. It is straightforward to show that each sequence in $E$ can be assumed to be a subsequence of $((x_{n},y_{n}))_{n\in \mathbb{N}}$ and this shows that $\bar{E}\bigcap ((\alpha X\times \alpha X)\setminus (X\times X))=\{(\omega,\omega)\}$. So $E\in \mathcal{E}$. Let $L_{1}=\{x_{n}\mid n\in \mathbb{N}\}\subseteq A$ and $L_{2}=\{y_{n}\mid n\in \mathbb{N}\}\subseteq B$. We have $L_{1}\lambda L_{2}$ and it shows that $A$ and $B$ are not asymptotically disjoint. Now assume that $A$ and $B$ are two subsets of $X$ such that they are not asymptotically disjoint. Let $L_{1}$ and $L_{2}$ be two unbounded and asymptotically alike subsets of $A$ and $B$ respectively. There is a $E\in \mathcal{E}$ such that $L_{1}\subseteq E(L_{2})$ and $L_{2}\subseteq E(L_{1})$. Let $\omega \in \bar{L_{1}}\bigcap (\alpha X\setminus X)$ and $(x_{n})_{n\in \mathbb{N}}$ be a sequence in $L_{1}$ and $x_{n}\rightarrow \omega$. For each $n\in \mathbb{N}$ choose $y_{n}\in L_{2}$ such that $(x_{n},y_{n})\in E$. Since $E\in \mathcal{E}$, $y_{n}\rightarrow \omega$. It shows that $\omega \in \bar{L_{1}}\bigcap \bar{L_{2}}\bigcap (\alpha X\setminus X)$. Thus $\bar{A}\bigcap \bar{B}\bigcap (\alpha X\setminus X)\neq \emptyset$.\\
Now we show this proposition claim. Let $A$ and $B$ be two asymptotically disjoint subsets of $X$. So $\bar{A}\bigcap \bar{B}\bigcap (\alpha X\setminus X)= \emptyset$. Since $\alpha X$ is a normal topological space there is a map $f:\alpha X\rightarrow [0,1]$ such that $f(\bar{A}\bigcap (\alpha X\setminus X)=0$ and $f(\bar{B}\bigcap (\alpha X\setminus X)=1$. Let $X_{1}=f^{-1}([\frac{1}{2},1])\bigcap X$ and $X_{2}=f^{-1}([0,\frac{1}{2}])\bigcap X$. By what we proved first here it is straightforward to show that $A$ and $B$ are asymptotically disjoint from $X_{1}$ and $X_{2}$ respectively.
\end{proof}
\begin{definition}
Let $(X,\mathcal{T})$ be a topological space and $\lambda$ be an AS.R. compatible with $\mathcal{T}$. We say that $\lambda$ is \emph{proper} if each bounded subset of $X$ has a compact closure.
\end{definition}
It is straightforward to show that a proper coarse structure admits a proper AS.R. It is an immediate result of the definition that if there exists a proper AS.R. on a topological space $X$, then $X$ is a locally compact topological space.
\begin{proposition}
Suppose that $\lambda$ is a proper and asymptotically connected AS.R. on a topological space $X$. Then a subset $A$ of $X$ is bounded if and only if $\bar{A}$ is compact.
\end{proposition}
\begin{proof}
The "only if" part is a part of the definition. Suppose that $A$ is a subset of $X$ with compact closure. We cover $\bar{A}$ with the $U_{i}$, $i\in \{1,...,n\}$, such that each $U_{i}$ is an asymptotic neighbourhood of some $a_{i}\in \bar{A}$. We have $(\bigcup_{i=1}^{n}U_{i})\lambda \{a_{1},...,a_{n}\}$ by (i) of \ref{2setare}. Also we have $\{a_{1},...,a_{n}\}\lambda a_{1}$ by asymptotic connectedness of $\lambda$ so \ref{4setare} leads to $A\lambda a_{1}$.
\end{proof}
From now on we assume that all AS.R. spaces are asymptotically connected.
\begin{definition}\label{7setare}
Let $(X,\mathcal{T})$ be a topological space and $\lambda$ be an AS.R. compatible with $\mathcal{T}$. For two nonempty subsets $A$ and $B$ of $X$ define $A\sim B$ if $A=B$ or $A$ and $B$ are unbounded asymptotically alike subsets of $X$. The relation $\sim$ is an equivalence relation on the family of all nonempty subsets of $X$. Let $\gamma X$ denotes the family of all closed ultrafilters on $X$ and $\mathcal{F}_{1},\mathcal{F}_{2}\in \gamma X$. Define $\mathcal{F}_{1}\approx \mathcal{F}_{2}$ if for any $A\in \mathcal{F}_{1}$ and $B\in \mathcal{F}_{2}$ there are $L_{1}\subseteq A$ and $L_{2}\subseteq B$ such that $L_{1}\sim L_{2}$. We denote the equivalence class of $\mathcal{F}\in \gamma X$ by $[\mathcal{F}]$.
\end{definition}
\begin{lemma}\label{zzz}
Let $(X,\lambda)$ be an AS.R. space. If $A$ and $B$ are asymptotically disjoint subsets of $X$ and $A\lambda C$ and $B\lambda D$ for some $C,D\subseteq X$, then $C$ and $D$ are asymptotically disjoint too.
\end{lemma}
\begin{proof}
It is an immediate consequence of \ref{6setare}.
\end{proof}
\begin{proposition}
Let $(X,\mathcal{T})$ be a topological space and let $\lambda$ be an AS.R. compatible with $\mathcal{T}$. If $(X,\lambda)$ is an asymptotically normal AS.R. space then the relation $\approx$ defined in \ref{7setare} is an equivalence relation.
\end{proposition}
\begin{proof}
The relation $\approx$ is obviously symmetric and reflexive. Suppose that $\mathcal{F}_{1}\approx \mathcal{F}_{2}$ and $\mathcal{F}_{2}\approx \mathcal{F}_{3}$ we claim $\mathcal{F}_{1}\approx \mathcal{F}_{3}$. Suppose that, on the contrary to our claim, there are disjoint sets $A\in \mathcal{F}_{1}$ and $C\in \mathcal{F}_{3}$ such that they are asymptotically disjoint. So $A$ and $C$ are not in $\mathcal{F}_{2}$. Choose $B\in \mathcal{F}_{2}$ such that $B\bigcap (A\bigcup C)=\emptyset$. Since $(X,\lambda)$ is asymptotically normal there are $X_{1}\subseteq X$ and $X_{2}\subseteq X$ such that $X_{1}\bigcup X_{2}=X$ and they are asymptotically disjoint from $A$ and $C$ respectively. Let $B_{1}=B\bigcap X_{1}$ and $B_{2}=B\bigcap X_{2}$. By compatibility and \ref{zzz}, $\bar{B_{1}}$ and $\bar{B_{2}}$ are asymptotically disjoint from $A$ and $C$ respectively. Since $\mathcal{F}_{2}$ is a closed ultrafilter and $B=\bar{B_{1}}\bigcup \bar{B_{2}}$ so $\bar{B_{1}}\in \mathcal{F}_{2}$ or $\bar{B_{2}}\in \mathcal{F}_{2}$ which contradicts $\mathcal{F}_{1}\approx \mathcal{F}_{2}$ or $\mathcal{F}_{2}\approx \mathcal{F}_{3}$ respectively.
\end{proof}
Let us recall that for an open subset $U$ of a topological space $X$, $U^{*}$ is the family of all closed ultrafilters on $X$ such that $U$ contains some elements of them.
\begin{proposition}\label{bbbb}
Let $X$ be a normal topological space and let $\lambda$ be a compatible and asymptotically normal AS.R. on $X$. Then the set $R=\{(\mathcal{F}_{1},\mathcal{F}_{2})\in \gamma X\times \gamma X\mid \mathcal{F}_{1}\approx \mathcal{F}_{2}\}$ is closed in $\gamma X\times \gamma X$.
\end{proposition}
\begin{proof}
Suppose that $(\mathcal{F}_{1},\mathcal{F}_{2})$ is not in $R$. So there are disjoint sets $A\in \mathcal{F}_{1}$ and $B\in \mathcal{F}_{2}$ such that they are also asymptotically disjoint. We choose asymptotic neighbourhoods $A\subseteq U$ and $B\subseteq V$ such that $U\bigcap V=\emptyset$. So $\mathcal{F}_{1}\in U^{*}$ and $\mathcal{F}_{2}\in V^{*}$. Now assume that $\mathcal{H}_{1}\in U^{*}$ and $\mathcal{H}_{2}\in V^{*}$. Thus there are $D_{1}\in \mathcal{H}_{1}$ and $D_{2}\in \mathcal{H}_{2}$ such that $D_{1}\subseteq  U$ and $D_{2}\subseteq V$. So $D_{1}\bigcup A\in \mathcal{H}_{1}$ and $D_{2}\bigcup B\in \mathcal{H}_{2}$. Also we have $(D_{1}\bigcup A)\lambda A$ and $(D_{2}\bigcup B)\lambda B$ by \ref{sss}. By \ref{zzz}, $D_{1}\bigcup A$ and $D_{2}\bigcup B$ are asymptotically disjoint. Therefore the open neighbourhood $U^{*}\times V^{*}$ of $(\mathcal{F}_{1},\mathcal{F}_{2})$ is disjoint from $R$.
\end{proof}
 Let $\lambda$ be a compatible AS.R. on a Hausdorff topological space $X$. Let us recall that for a point $x\in X$, $\sigma_{x}$ denotes the family of all closed subset of $X$ that contains $x$ and the map $\sigma: X\rightarrow \gamma X$ defined by $\sigma(x)=\sigma_{x}$ is a topological embedding. For two points $x,y\in X$, it is straightforward to show that $\sigma_{x} \approx \sigma_{y}$ if and only if $x=y$. Thus the map $\phi :X\rightarrow \frac{\gamma X}{\approx}$ defined by $\phi (x)=[\sigma_{x}]$ is one to one.
\begin{corollary}
Let $X$ be a normal topological space and let $\lambda$ be a proper and asymptotically normal AS.R. on $X$. Then $\frac{\gamma X}{\approx}$ is a Hausdorff compactification of $X$.
\end{corollary}
\begin{proof}
Since $\gamma X$ is compact, its quotient $\frac{\gamma X}{\approx}$ is compact too. By \ref{bbbb} $\frac{\gamma X}{\approx}$ is Hausdorff. It suffices to show that $\phi :X\rightarrow \frac{\gamma X}{\approx}$ is a topological embedding. Let $\pi:\gamma X\rightarrow \frac{\gamma X}{\approx}$ be the quotient map. Since $\phi=\pi \circ \sigma$, $\phi$ is a continuous map. Suppose that $U\subseteq X$ is an open set and $[\sigma_{x}]\in \phi (U)$. By \ref{sss} we can choose an asymptotic neighbourhood $W$ of $x$ such that $W\subseteq U$. It is easy to verify $\pi^{-1}(\phi (W))=W^{*}$. Thus $\phi (W)$ is open in $\frac{\gamma X}{\approx}$ and we have $[\sigma_{x}]\in \phi (W)\subseteq \phi (U)$. Therefore $\phi$ is a topological embedding and $\phi (X)$ is open in $\frac{\gamma X}{\approx}$.
\end{proof}

\begin{proposition}\label{8setare}
Let $(X,\mathcal{T})$ be a topological space and let $\lambda$ be an AS.R. compatible with $\mathcal{T}$. Suppose that $(x_{\alpha})_{\alpha \in I}$ and $(y_{\alpha})_{\alpha \in I}$ are two nets in $X$. Let $T_{\beta}=\{x_{\alpha}\mid \alpha \geq \beta \}$ and $S_{\beta}=\{y_{\alpha}\mid \alpha \geq \beta \}$. If $T_{\beta}\lambda S_{\beta}$ for all $\beta \in I$ and $\sigma_{x_{\alpha}}\rightarrow \mathcal{F}_{1}$ and $\sigma_{y_{\alpha}}\rightarrow \mathcal{F}_{2}$ for some $\mathcal{F}_{1},\mathcal{F}_{2}\in \gamma X\setminus\sigma(X)$ then $\mathcal{F}_{1}\approx \mathcal{F}_{2}$.
\end{proposition}
\begin{proof}
Suppose that $A\in \mathcal{F}_{1}$ and $B\in \mathcal{F}_{2}$. We choose asymptotic neighbourhoods $U$ and $V$ of $A$ and $B$ respectively. So $\mathcal{F}_{1}\in U^{*}$ and $\mathcal{F}_{2}\in V^{*}$. Since $\sigma_{x_{\alpha}}\rightarrow \mathcal{F}_{1}$  and $\sigma_{y_{\alpha}}\rightarrow \mathcal{F}_{2}$ there are $\alpha , \beta \in I$ such that $T_{\alpha}\subseteq U$ and $S_{\beta}\subseteq V$. Let $\alpha , \beta \leq \gamma$ so $T_{\gamma}\subseteq U$ and $S_{\gamma} \subseteq V$ and this leads to $T_{\gamma} \lambda L_{1}$ and $S_{\gamma} \lambda L_{2}$ for some $L_{1}\subseteq A$ and $L_{2}\subseteq B$ by \ref{6setare}.
\end{proof}
 A compactification $\bar{X}$ of a proper coarse space $(X,\mathcal{E})$ is said to be a \emph{coarse compactification} of $X$ when, if $E\in \mathcal{E}$ and $(x_{\alpha},y_{\alpha})_{\alpha \in I}$ is a convergent net in $E$, then $x_{\alpha}\rightarrow \omega$ for $\omega \in \bar{X}\setminus X$ yields $y_{\alpha}\rightarrow \omega$ (\cite{Roe}).
\begin{corollary}\label{oooo}
Let $X$ be a normal topological space and let $\lambda$ be an AS.R. associated to a proper coarse structure $\mathcal{E}$ on $X$. Suppose  that $\lambda$ is asymptotically normal. Then $\frac{\gamma X}{\approx}$ is a coarse compactification.
\end{corollary}
\begin{proof}
Let $(x_{\alpha},y_{\alpha})_{\alpha \in I}$ be a convergent net in $E\in \mathcal{E}$. Assume that $[\sigma_{x_{\alpha}}]\rightarrow [\mathcal{F}_{1}]$ and $[\sigma_{y_{\alpha}}]\rightarrow [\mathcal{F}_{2}]$ for $[\mathcal{F}_{1}],[\mathcal{F}_{2}]\in \frac{\gamma X}{\approx} \setminus\phi(X)$. Suppose that $\sigma_{x_{\alpha_{i}}}$ is a convergent subnet of $\sigma_{x_{\alpha}}$ and $\sigma_{y_{\alpha_{i_{k}}}}$ is a convergent subnet of $\sigma_{y_{\alpha_{i}}}$. If $\sigma_{x_{\alpha_{i_{k}}}}\rightarrow \mathcal{H}_{1}$ and  $\sigma_{y_{\alpha_{i_{k}}}}\rightarrow \mathcal{H}_{2}$ we have $\mathcal{H}_{1}\approx \mathcal{F}_{1}$ and $\mathcal{H}_{2}\approx \mathcal{F}_{2}$. Thus by \ref{8setare} we have $\mathcal{H}_{1}\approx \mathcal{H}_{2}$ and therefore $\mathcal{F}_{1}\approx \mathcal{F}_{2}$.
\end{proof}
\begin{corollary}\label{ss1}
Let $X$ be a normal topological space and let $\mathcal{E}$ be a proper coarse structure on $X$. Assume that the AS.R. associated to $\mathcal{E}$ is asymptotically normal. Then the identity map $i:X\rightarrow X$ extends uniquely to a continuous map of $hX$ into $\frac{\gamma X}{\approx}$.
\end{corollary}
\begin{proof}
It is an immediate consequence of previous corollary and 2.39 of \cite{Roe}.
\end{proof}
\begin{proposition}\label{ss2}
Assume the hypotheses of corollary \ref{ss1}. Each Higson function $f:X\rightarrow \mathbb{C}$ has a unique extension $\bar{f}:\frac{\gamma X}{\approx}\rightarrow \mathbb{C}$.
\end{proposition}
\begin{proof}
Let $f:X\rightarrow \mathbb{R}$ be a Higson function and $\hat{f}:\gamma X\rightarrow \mathbb{C}$ be its extension to $\gamma X$. Suppose that $\mathcal{F}_{1},\mathcal{F}_{2}\in \gamma X\setminus X$ and $\mathcal{F}_{1}\approx \mathcal{F}_{2}$. Let $\hat{f}(\mathcal{F}_{1})=x_{1}$ and $\hat{f}(\mathcal{F}_{2})=x_{2}$. Assume that $x_{1}\neq x_{2}$. Let $\delta=\frac{\mid x_{1}-x_{2}\mid}{4}$. Then $\hat{f}^{-1}(B(x_{1},\delta))$ and $\hat{f}^{-1}(B(x_{2},\delta))$ are open sets containing $\mathcal{F}_{1}$ and $\mathcal{F}_{2}$ respectively, so there are open sets $U\subseteq X$ and $V\subseteq X$ such that $\mathcal{F}_{1}\in U^{*}\subseteq \hat{f}^{-1}(B(x_{1},\delta))$ and $\mathcal{F}_{2}\in V^{*}\subseteq \hat{f}^{-1}(B(x_{2},\delta)$. Thus there are $A\in \mathcal{F}_{1}$ and $B\in\mathcal{F}_{2}$ such that $A\subseteq U$ and $B\subseteq V$. Since $\mathcal{F}_{1}\approx \mathcal{F}_{2}$ there are unbounded and asymptotically alike subsets $L_{1}\subseteq A$ and $L_{2}\subseteq B$. So there is $E\in \mathcal{E}$ such that $L_{1}\subset E(L_{2})$ and $L_{2}\subseteq E(L_{1})$. Since $f$ is a Higson function there is a compact $K\subseteq X$ such that $\mid f(x)-f(y) \mid <\delta$ for all $(x,y)\in E\setminus K\times K$. Let $x\in L_{1}\setminus K$ and $y\in L_{2}\setminus K$ so $\sigma_{x}\in U^{*}$ and $\sigma_{y}\in V^{*}$. It leads to $\hat{f}(\sigma_{x})=f(x)\in B(x_{1},\delta)$ and $\hat{f}(\sigma_{y})=f(y)\in B(x_{2},\delta)$ so $$\mid x_{2}-x_{1}\mid \leq \mid x_{2}-f(y) \mid + \mid f(x)-f(y) \mid+\mid x_{1}-f(x) \mid < 3\delta$$ Thus $\mid x_{2}-x_{1}\mid <\frac{3\mid x_{1}-x_{2}\mid}{4}$, a contradiction. Therefore $x_{1}=x_{2}$. Define $\bar{f}:\frac{\gamma X}{\approx}\rightarrow \mathbb{C}$ by $\bar{f}([\mathcal{F}])=\hat{f}(\mathcal{F})$. The map $\bar{f}$ is well defined and since $\bar{f}\circ \pi=\hat{f}$, it is continuous.
\end{proof}
\begin{corollary}\label{9setare}
Assume the hypotheses of corollary \ref{ss1}. Then $hX$ and $\frac{\gamma X}{\approx}$ are homeomorphic.
\end{corollary}
\begin{proof}
The proposition \ref{ss2} shows that the identity map $i:X\rightarrow X$ extends uniquely to a map from $\frac{\gamma X}{\approx}$ to $hX$. Thus \ref{ss1} shows that $hX$ and $\frac{\gamma X}{\approx}$ are homeomorphic.
\end{proof}
Suppose that $(X,\mathcal{T})$ is a topological space and $\lambda$ is a proper and asymptotically normal AS.R. on it. We call $\frac{\gamma X}{\approx}$ the \emph{asymptotic compactification} of $X$. We also call $\nu X=\frac{\gamma X}{\approx}\setminus\phi(X)$ the \emph{asymptotic corona} of $X$. For an AS.R. associated to a proper coarse structure $\mathcal{E}$ on $X$, \ref{9setare} shows that $\nu X$ is homeomorphic with Higson corona.
\begin{example}
Let $(X,d)$ be a metric space. For two subsets $A$ and $B$ of $X$, define $A\lambda B$ if $A$ and $B$ are both unbounded or $A$ and $B$ are both bounded. The relation $\lambda$ is a proper AS.R. on $(X,d)$. Two subsets $A$ and $B$ of $X$ are asymptotically disjoint if and only if $A$ is bounded or $B$ is bounded. For a bounded subset $A\subseteq X$, let $X_{1}=X\setminus A$ and $X_{2}=A$ then $X_{1}$ is asymptotically disjoint from $A$ and $X_{2}$ is asymptotically disjoint from $B$, for all $B\subseteq X$. Thus $(X,\lambda)$ is asymptotically normal AS.R. space. It is straightforward to show that $\mathcal{F}_{1}\approx \mathcal{F}_{2}$, for all $\mathcal{F}_{1},\mathcal{F}_{2}\in \gamma X\setminus \sigma(X)$. Therefore the asymptotic compactification of $X$ is the one point compactification of $(X,d)$.
\end{example}
\begin{example}
Suppose that $\lambda$ is the AS.R. introduced in \ref{lll} on $\mathbb{R}$. Since all two unbounded subsets of $\mathbb{R}$ with respect to $\lambda$, are asymptotically alike, two subsets $A$ and $B$ of $\mathbb{R}$ are asymptotically disjoint if and only if $A$ is bounded or $B$ is bounded with respect to $\lambda$. Let $A$ be a bounded subset of $\mathbb{R}$ with respect to $\lambda$. So $A\subseteq (r,+\infty)$ for some $r\in \mathbb{R}$. Let $X_{1}=(-\infty,r)$ and $X_{2}=(r,+\infty)$. The sets $X_{1}$ and $A$ are asymptotically disjoint and $X_{2}$ is asymptotically disjoint from $B$, for all $B\subseteq X$. Thus $(X,\lambda)$ is an asymptotically normal AS.R. space. At each point $x\in \mathbb{R}$ other than the origin assume the usual neighbourhood basis at $x$. At the origin let $\mathcal{B}=\{(-\epsilon,+\epsilon)\bigcup (n,+\infty)\mid n\in \mathbb{N},\epsilon>0 \}$ be the neighbourhood basis. Let $\mathcal{T}$ be the corresponding topology on $\mathbb{R}$. It is easy to show that $\lambda$ is a proper AS.R. on $(\mathbb{R},\mathcal{T})$ and $(\mathbb{R},\mathcal{T})$ is a normal topological space. For all $\mathcal{F}_{1},\mathcal{F}_{2}\in \gamma X\setminus \sigma(X)$ we have $\mathcal{F}_{1}\approx \mathcal{F}_{2}$. Therefore the asymptotic compactification of $(\mathbb{R},\lambda)$ is the one point compactification of $(\mathbb{R},\mathcal{T})$.
\end{example}
\begin{proposition}\label{map}
Let $X$ and $Y$ be two topological spaces equipped with two proper and asymptotically normal AS.R.s. For every continuous AS.R. mapping $f:X\rightarrow Y$ there exists a unique continuous extension $\tilde{f}:\frac{\gamma X}{\approx}\rightarrow \frac{\gamma Y}{\approx}$ which sends $\nu X$ to $\nu Y$.
\end{proposition}
\begin{proof}
 For $\mathcal{F}\in \gamma X$, define $f_{*}(\mathcal{F})=\{A\subseteq Y\mid A\, is\, closed\, and\, f^{-1}(A)\in \mathcal{F}\}$. Let $\hat{f}(\mathcal{F})$ be a unique closed ultrafilter that contains $f_{*}(\mathcal{F})$ (\cite{Wil} 16K). The map $\hat{f}:\gamma X\rightarrow \gamma Y$ is a continuous extension of $f$ (\cite{Wil} 19K). Assume that $\mathcal{F}\in \gamma X\setminus\sigma(X)$ and $\hat{f}(\mathcal{F})=\sigma_{y}$ for some $y\in Y$. So for all $A\in f_{*}(\mathcal{F})$ we have $y\in A$. Let $U\subseteq Y$ be an asymptotic neighbourhood of $y$. We have $A=(A\setminus U)\bigcup (\overline{A\bigcap U})$. Since $y$ is not in $A\setminus U$ and $f_{*}(\mathcal{F})$ is a prime closed filter so $\overline{A\bigcap U}\in f_{*}(\mathcal{F})$. Since $f$ is an AS.R. mapping thus $f^{-1}(\overline{A\bigcap U})$ is bounded and it contradicts $\mathcal{F}\in \gamma X\setminus\sigma(X)$. Thus $\hat{f}$ sends $\gamma X\setminus \sigma(X)$ to $\gamma Y\setminus \sigma(Y)$. Suppose that $\mathcal{F}_{1}\approx \mathcal{F}_{2}$. Let $C\in \hat{f}(\mathcal{F}_{1})$ and $A\in f_{*}(\mathcal{F}_{1})$. Assume that $U$ is an asymptotic neighbourhood of $C$. Since $A=(\overline{A\bigcap U})\bigcup (A\setminus U)$ and $(A\setminus U)\bigcap U= \emptyset$ so $\overline{A\bigcap U}\in f_{*}(\mathcal{F}_{1})$. Similarly one can show that for $D\in \hat{f}(\mathcal{F}_{2})$ and $B\in f_{*}(\mathcal{F}_{2})$ we have $\overline{B\bigcap V}\in f_{*}(\mathcal{F}_{2})$ for some asymptotic neighbourhood $V$ of $D$. So $f^{-1}(\overline{A\bigcap U})\in \mathcal{F}_{1}$ and $f^{-1}(\overline{B\bigcap V})\in \mathcal{F}_{2}$. Thus there are unbounded and asymptotically alike subsets $L_{1}\subseteq f^{-1}(\overline{A\bigcap U})$ and $L_{2}\subseteq f^{-1}(\overline{B\bigcap V})$. Since $f$ is an AS.R. mapping $f(L_{1})$ and $f(L_{2})$ are unbounded and asymptotically alike subsets of $\overline{A\bigcap U}$ and $\overline{B\bigcap V}$ respectively. Since $\lambda$ is compatible with the topology, \ref{6setare} shows that $C$ and $D$ are not asymptotically disjoint. Thus $\hat{f}(\mathcal{F}_{1})\approx \hat{f}(\mathcal{F}_{2})$. Therefore $\tilde{f}:\frac{\gamma X}{\approx} \rightarrow \frac{\gamma Y}{\approx}$ defined by $\tilde{f}([\mathcal{F}])=[\hat{f}(\mathcal{F})]$ is well defined. We have $\tilde{f}\circ \pi=\pi^{\prime} \circ \hat{f}$, where $\pi: \gamma X\rightarrow \frac{\gamma X}{\approx}$ and $\pi^{\prime}:\gamma Y \rightarrow \frac{\gamma Y}{\approx}$ are quotient maps. So $\tilde{f}$ is continuous and since $\hat{f}$ sends $\gamma X\setminus \sigma(X)$ to $\gamma Y\setminus \sigma(Y)$ it sends $\nu X$ to $\nu Y$.
\end{proof}

In the following propositions $\bar{A}$ denotes the closure of $A\subseteq X$ in $\frac{\gamma X}{\approx}$.
\begin{proposition}\label{10setare}
Let $X$ be a normal topological space and let $\mathcal{E}$ be a proper coarse structure on $X$. Assume that the AS.R. associated to $\mathcal{E}$ is asymptotically normal. If $A$ and $B$ are two asymptotically alike subsets of $X$ then $\bar{A}\bigcap \nu X=\bar{B}\bigcap \nu X$.
\end{proposition}
\begin{proof}
Let $[\mathcal{F}]\in \bar{A}\bigcap \nu X$. Let us denote by $D^{\prime}$ the closure of $D\subseteq X$ in $\gamma X$. Since $\pi: \gamma X\rightarrow \frac{\gamma X}{\approx}$ is a closed map so $(\bar{A}\bigcap \nu X)\subseteq \pi(A^{\prime})\bigcap \nu X$. Thus there is a ultrafilter $\mathcal{G}\in A^{\prime}$ such that $\mathcal{F}\approx \mathcal{G}$. There is a net $(x_{\alpha})_{\alpha \in I}$ in $A$ such that $\sigma_{x_{\alpha}}\rightarrow \mathcal{G}$. Since $A$ and $B$ are asymptotically alike, $A\subseteq E(B)$ and $B\subseteq E(A)$ for some $E\in \mathcal{E}$. For each $\alpha \in I$ we choose $y_{\alpha}\in B$ such that $(x_{\alpha},y_{\alpha})\in E$. The net $(\sigma_{y_{\alpha}})_{\alpha \in I}$ has a convergent subnet $(\sigma_{y_{\alpha_{i}}})_{i\in J}$. So $\sigma_{y_{\alpha_{i}}}\rightarrow \mathcal{H}$ for some $\mathcal{H}\in B^{\prime}$. Two nets $(\sigma_{y_{\alpha_{i}}})_{i\in J}$ and $(\sigma_{x_{\alpha_{i}}})_{i\in J}$ satisfy the assumption of \ref{8setare}. Thus $\mathcal{G}\approx \mathcal{H}$ and it leads to $[\mathcal{F}]\in \bar{B}\bigcap \nu X$.
\end{proof}
\begin{corollary}\label{11setare}
Assume the hypotheses of \ref{10setare}. Two subsets $A$ and $B$ of $X$ are asymptotically disjoint if and only if $$(\bar{A}\bigcap \nu X)\bigcap (\bar{B}\bigcap \nu X)=\emptyset$$.
\end{corollary}
\begin{proof}
Suppose that $A,B\subseteq X$ and $(\bar{A}\bigcap \nu X)\bigcap (\bar{B}\bigcap \nu X)=\emptyset$. Assume that, on the contrary, there are unbounded and asymptotically alike subsets $L_{1}\subseteq A$ and $L_{2}\subseteq B$. By \ref{10setare}, $(\bar{L_{1}}\bigcap \nu X)=(\bar{L_{2}}\bigcap \nu X)\neq \emptyset$. Since $\bar{L_{1}}\subseteq \bar{A}$ and $\bar{L_{2}}\subseteq \bar{B}$ so $(\bar{A}\bigcap \nu X)\bigcap (\bar{B}\bigcap \nu X)\neq \emptyset$, a contradiction.\\
To prove the converse assume that $A$ and $B$ are asymptotically disjoint. Let $[\mathcal{F}]\in \bar{A}\bigcap \nu X$. As in the previous proposition, let us denote by $D^{\prime}$ the closure of $D\subseteq X$ in $\gamma X$. So there is $\mathcal{G}\in A^{\prime}$ such that $\mathcal{G}\approx \mathcal{F}$ . Let $\mathcal{H}\in B^{\prime}$. The closures of $A$ and $B$ in topological space $X$ are in $\mathcal{G}$ and $\mathcal{H}$ respectively. Since $\lambda$ is an AS.R. compatible with the topology, $\mathcal{G}$ and $\mathcal{H}$ contain asymptotically disjoint sets. Therefore $[\mathcal{F}]$ is not in $\bar{B}\bigcap \nu X$.
\end{proof}
Now we will prove the converse of \ref{10setare} for metric spaces.
\begin{corollary}
Assume that $(X,d)$ is proper metric space. For two subsets $A$ and $B$ of $X$ if $\bar{A}\bigcap \nu X=\bar{B}\bigcap \nu X$ then $A$ and $B$ are asymptotically alike.
\end{corollary}
\begin{proof}
Suppose that $A$ and $B$ are not asymptotically alike. We can assume that (without loss of generality) for each $n\in \mathbb{N}$, $A$ is not a subset of $\textbf{B}(B,n)$. For each $n\in \mathbb{N}$ choose $a_{n}\in A$ such that $d(a_{n},B)\geq n$. Let $L=\{a_{n}\mid n\in \mathbb{N}\}$. Clearly $L$ and $B$ are asymptotically disjoint. Thus by \ref{11setare}, $(\bar{L}\bigcap \nu X)\bigcap (\bar{B}\bigcap \nu X)=\emptyset$. It is a contradiction, since $\bar{L}\bigcap \nu X\subseteq \bar{A}\bigcap \nu X$.
\end{proof}
In the previous section we showed that two different coarse structures may induce the same asymptotic resemblance relation. The following proposition shows that such coarse structures have the same Higson compactifications too.
\begin{proposition}
Let $\mathcal{E}_{1}$ and $\mathcal{E}_{2}$ be two proper and connected coarse structures on a Hausdorff topological space $X$. If $\mathcal{E}_{1}$ and $\mathcal{E}_{2}$ induce the same AS.R. $\lambda$ on $X$ then $(X,\mathcal{E}_{1})$ and $(X,\mathcal{E}_{2})$ have homeomorphic Higson compactification.
\end{proposition}
\begin{proof}
Let $C_{h_{1}}(X)$ and $C_{h_{2}}(X)$ denote the family of all Higson functions on $(X,\mathcal{E}_{1})$ and $(X,\mathcal{E}_{2})$ respectively. Suppose that $f\in C_{h_{1}}(X)$. Let $E\in \mathcal{E}_{2}$ and $\epsilon >0$. We have $\pi_{1}(E)\lambda \pi_{2}(E)$, where $\pi_{1}$ and $\pi_{2}$ are projection functions on first and second coordinate respectively. So there is a $F\in \mathcal{E}_{1}$ such that $\pi_{E}\subseteq F(\pi_{2}(E))$ and $\pi_{2}(E)\subseteq F(\pi_{1}(E))$. Since $f\in C_{h_{1}}(X)$ there exists a compact subset $K$ of $X$ such that $\mid f(x)-f(y)\mid <\frac{\epsilon}{2}$ for all $(x,y)\in F\setminus K\times K$. Let $L=\overline{E(K)\bigcup E^{-1}(K)\bigcup K}$. Since $\mathcal{E}_{2}$ is proper so $L$ is a compact subset of $X$. Assume that $(x,y)\in E\setminus L\times L$. Suppose that $x$ does not belong to $L$ so it is not in $K$. Also since $x$ does not belong to $E^{-1}(K)$ so $y$ is not in $K$. There are $x^{\prime}\in \pi_{1}(E)$ and $y^{\prime}\in \pi_{2}(E)$ such that $(x,y^{\prime})\in F$ and $(x^{\prime},y)\in F$. Since $x$ and $y$ are not in $K$ so $(x,y^{\prime})$ and $(x^{\prime},y)$ are in $F\setminus K\times K$. Thus $\mid f(x)-f(y)\mid \leq \mid f(x)-f(y^{\prime})\mid +\mid f(x^{\prime})-f(y)\mid < \epsilon$. Similar arguments hold if $f$ does not belong to $F$. Therefore $f\in C_{h_{2}}(X)$ and it shows that $C_{h_{1}}(X)\subseteq C_{h_{2}}(X)$. Similarly one can shows that $C_{h_{2}}(X)\subseteq C_{h_{1}}(X)$.
\end{proof}
\section{Asymptotic compactification and proximity}
Let $(X,\mathcal{T})$ be a topological space and $\lambda$ be an AS.R. compatible with $\mathcal{T}$. Suppose that the relation $\sim$ is as \ref{7setare}. For two subsets $A$ and $B$ of $X$, define $A\delta_{\lambda} B$ if there are $L_{1}\subseteq \bar{A}$ and $L_{2}\subseteq \bar{B}$ such that $L_{1}\sim L_{2}$.
\begin{proposition}
Let $(X,\mathcal{T})$ be a normal topological space and let $\lambda$ be a proper and asymptotically normal AS.R. on $X$. Then $\delta_{\lambda}$ is a separated proximity on $X$ and it is compatible with $\mathcal{T}$.
\end{proposition}
\begin{proof}
The relation $\delta_{\lambda}$ clearly satisfies properties (i), (ii) and (iii) of \ref{prox}. Assume that $A\delta_{\lambda} (B\bigcup C)$. So there are $L_{1}\subseteq \bar{A}$ and $L_{2}\subseteq \overline{B\bigcup C}$ such that $L_{1}\sim L_{2}$. If $L_{1}=L_{2}$ then $\bar{A}\bigcap(\overline{B\bigcup C})\neq \emptyset$ and it leads to $A\delta_{\lambda} B$ or $A\delta_{\lambda} C$ clearly. If $L_{1}$ and $L_{2}$ are two unbounded asymptotically alike subsets of $X$, then $L_{2}\bigcap \bar{B}$ or $L_{2}\bigcap \bar{C}$ should be unbounded. Assume that $L_{2}\bigcap \bar{B}$ is unbounded. So there is unbounded subset $L_{3}\subseteq L_{1}$ such that $(L_{2}\bigcap \bar{B})\lambda L_{3}$ by \ref{6setare}. Thus $A\delta_{\lambda} B$. If $A\delta_{\lambda} B$ it is straightforward to show that $A\delta_{\lambda} (B\bigcup C)$ for all $C\subseteq X$. Now assume that $A,B\subseteq X$ and $A\bar{\delta_{\lambda}} B$. So $\bar{A}$ and $\bar{B}$ are two disjoint and asymptotically disjoint subsets of $X$. We choose $X_{1}\subseteq X$and $X_{2}\subseteq X$ such that $X=X_{1}\bigcup X_{2}$ and they are asymptotically disjoint from $\bar{A}$ and $\bar{B}$ respectively. Since $(X,\mathcal{T})$ is a normal topological space and $\lambda$ is compatible with $\mathcal{T}$ we can find asymptotic neighbourhoods $\bar{X_{1}}\bigcap \bar{A}\subseteq U$ and $\bar{X_{2}}\bigcap \bar{B}\subseteq V$ such that $\bar{U}\bigcap \bar{B}=\emptyset$ and $\bar{V}\bigcap \bar{A}=\emptyset$. Let $E=(X_{1}\setminus U)\bigcup V$. Since $X_{1}$ and $\bar{A}$ are asymptotically disjoint, $\bar{X_{1}}\bigcap \bar{A}$ is bounded and it shows that $U$ is bounded. Similarly $V$ is bounded. Thus $\bar{A}$ and $\bar{E}$ are disjoint and they are asymptotically disjoint too since $V$ is bounded. Therefore $A\bar{\delta_{\lambda}} E$. Similarly one can show that $\bar{B}$ and $\overline{X\setminus E}\subseteq \overline{(X_{2}\setminus V)\bigcup U}$ are disjoint and asymptotically disjoint and it leads to $(X\setminus E)\bar{\delta_{\lambda}} B$. Since $\lambda$ is proper one can easily verify that $\delta_{\lambda}$ is compatible with the topology.
\end{proof}
Let us recall that on a separated proximity space $(X,\delta)$, $\mathfrak{X}$ denotes the family of all clusters in $X$. For  two subsets $\mathfrak{M}$ and $\mathfrak{N}$ of $X$, $\mathfrak{M}\delta^{*} \mathfrak{N}$ means that if $A\subseteq X$ absorbs $\mathfrak{M}$ and $B\subseteq X$ absorbs $\mathfrak{N}$ then $A\delta B$. A subset $A$ of $X$ absorbs $\mathfrak{M}\subseteq \mathcal{X}$ means that $A\in \mathcal{C}$ for all $\mathcal{C} \in \mathfrak{M}$. The proximity space $(\mathfrak{X},\delta^{*})$ is called the Smirnov compactification of $X$.
\begin{proposition}
Let $(X,\mathcal{T})$ be a normal topological space and let $\lambda$ be a proper and asymptotically normal AS.R. on $X$. Then $\frac{\gamma X}{\approx}$ and the Smirnov compactification $(\mathfrak{X},\delta_{\lambda}^{*})$ are homeomorphic.
\end{proposition}
\begin{proof}
Let $\mathcal{F}\in \gamma X$ and let $\tilde{\mathcal{F}}=\{A\subseteq X\mid A\delta_{\lambda} B\,for\,all\,B\in \mathcal{F}\}$. The family $\tilde{\mathcal{F}}$ is a cluster in $X$ (\cite{Nai} theorem 5.8). Define $\psi: \frac{\gamma X}{\approx}\rightarrow \mathfrak{X}$ by $\psi([\mathcal{F}])=\tilde{\mathcal{F}}$ for all $\mathcal{F}\in \gamma X$. For $\mathcal{F},\mathcal{G}\in \gamma X$ if $\mathcal{F}\approx \mathcal{G}$ then $A\delta_{\lambda} B$ for all $A\in \mathcal{F}$ and all $B\in \mathcal{G}$, so $\tilde{\mathcal{F}}=\tilde{\mathcal{G}}$. Thus the map $\psi$ is well defined. It is straightforward to show that $\psi$ is one to one and by using 5.8 of \cite{Nai} one can easily shows that it is surjective too. Suppose that $\mathfrak{M} \subseteq \gamma X$ and $\mathcal{F}\in \bar{\mathfrak{M}}$. Let $A$ be a subset of $X$ such that $A\in \psi(\mathcal{G})$ for all $\mathcal{G}\in \mathfrak{M}$. We claim that $A\in \psi(\mathcal{F})$. Suppose that, contrary to our claim, $A$ is not in $\psi(\mathcal{F})$. So there exists $B\in \mathcal{F}$ such that $\bar{A}$ and $\bar{B}$ are disjoint and asymptotically disjoint. We choose asymptotic neighbourhood $\bar{B}\subseteq U$ such that $\bar{A}\bigcap \bar{U}=\emptyset$. The set $U^{*}$ is an open subset of $\gamma X$ containing $\mathcal{F}$. Thus there is $\mathcal{G}\in \mathfrak{M}$ and $C\in \mathcal{G}$ such that $C\subseteq U$. It shows that $\bar{C}$ and $\bar{A}$ are disjoint and asymptotically disjoint. So $A\bar{\delta_{\lambda}} C$ and it contradicts $A\in \psi(\mathcal{G})$. Therefore $\psi(\mathcal{F})\in \overline{\psi(\mathfrak{M})}$. It shows that $\psi\circ \pi$ is continuous, where $\pi:\gamma X\rightarrow \frac{\gamma X}{\approx}$ is the quotient map. So $\psi$ is continuous and Since $\frac{\gamma X}{\approx}$ is compact and Hausdorff, it is a homeomorphism too.
\end{proof}

\section{Asymptotic dimension}
Let $\mathcal{U}$ be a family of subsets of a set $X$ and let $S_{\mathcal{U}}=\bigcup_{U\in \mathcal{U}}U\times U$. For two subsets $A$ and $B$ of $X$, define $A\sim_{\mathcal{U}}B$ if $A\subseteq S_{\mathcal{U}}(B)$ and $B\subseteq S_{\mathcal{U}}(A)$.
\begin{definition}
We call a family $\mathcal{U}$ of subsets of an AS.R. space $(X,\lambda)$ \emph{uniformly bounded}, if\\
i) each $U\in \mathcal{U}$ is bounded.\\
ii) $A\sim_{\mathcal{U}}B$ implies $A\lambda B$, for all $A,B\subseteq X$.
\end{definition}
The following proposition shows that if $\lambda$ is the AS.R. associated to a metric $d$ on a set $X$, then the above definition coincides with uniformly boundedness with respect to $d$.
\begin{proposition}\label{uni}
Let $(X,d)$ be a metric space and let $\lambda$ be the AS.R. associated to $d$. A family $\mathcal{U}$ of subsets of $X$ is uniformly bounded if and only if there is $k>0$ such that $\operatorname{diam}(U)<k$ for all $U\in \mathcal{U}$.
\end{proposition}
\begin{proof}
The "if" part is easy to verify. To prove the converse assume that, on the contrary, for each $n\in \mathbb{N}$ there are $U_{n}\in \mathcal{U}$ and $x_{n},y_{n}\in U_{n}$ such that $d(x_{n},y_{n})>n$. For each subset $I\subseteq \mathbb{N}$ we have $A_{I}=\{x_{i}\mid i\in I\}\sim_{\mathcal{U}} B_{I}=\{y_{i}\mid i\in I\}$ so $A_{I}\lambda B_{I}$. Thus the sequences $(x_{n})_{n\in \mathbb{N}}$ and $(y_{n})_{n\in \mathbb{N}}$ satisfy the hypothesis of \ref{hhhh}, a contradiction.
\end{proof}
Let us recall that for a family $\mathcal{M}$ of subsets of a set $X$, $\mu(\mathcal{M})$ denotes the multiplicity of $\mathcal{M}$ i.e the greatest number of elements of $\mathcal{M}$ that meets a point of $X$
\begin{definition}
Let $(X,\lambda)$ be an AS.R. space. We say that $\operatorname{asdim}_{\lambda}X\leq n$ if for all uniformly bounded cover $\mathcal{U}$ of $X$ there is a uniformly bounded cover $\mathcal{V}$ for $X$ such that $\mathcal{U}$ refines $\mathcal{V}$ and $\mu(\mathcal{V})\leq n+1$. We say $\operatorname{asdim}_{\lambda}X=n$ if $\operatorname{asdim}_{\lambda}X\leq n$ and $\operatorname{asdim}_{\lambda}X\leq n-1$ is not true. We call $\operatorname{asdim}_{\lambda}X$ the \emph{asymptotic dimension} of an AS.R. space $(X,\lambda)$.
\end{definition}
The proposition \ref{uni} shows that on a metric space $(X,d)$ we have $\operatorname{asdim}X=\operatorname{asdim}_{\lambda}X$, where $\lambda$ is the AS.R. associated to $d$.\\
\begin{proposition}\label{asdim}
Let $(X,\lambda)$ be an AS.R. space and let $Y\subseteq X$. Then $\operatorname{asdim}_{\lambda_{Y}}Y\leq \operatorname{asdim}_{\lambda}X$
\end{proposition}
\begin{proof}
Suppose that $\operatorname{asdim}_{\lambda}X=n$. Let $\mathcal{V}$ be a uniformly bounded cover of $Y$. Assume that $\mathcal{U}=\mathcal{V}\bigcup_{x\in X\setminus Y}\{\{x\}\}$. If $A,B\subseteq X$ and $A\sim_{\mathcal{U}}B$, then $(A\bigcap (X\setminus Y))=(B\bigcap (X\setminus Y))$ and $(A\bigcap Y)\sim_{\mathcal{V}}(B\bigcap Y)$. So $(A\bigcap Y)\lambda (B\bigcap Y)$ and (i) of \ref{2setare} shows that $A\lambda B$. Thus $\mathcal{U}$ is a uniformly bounded cover of $X$. Let $\mathcal{W}$ be a uniformly bounded cover of $X$ such that $\mathcal{U}$ refines it and $\mu(\mathcal{W})\leq n+1$. The family $\mathcal{W}_{Y}=\{W\bigcap Y\mid W\in \mathcal{W}\}$ is a uniformly bounded cover of $Y$ and $\mathcal{V}$ refines it. Clearly $\mu(\mathcal{W}_{Y})\leq n+1$ so $\operatorname{asdim}_{\lambda_{Y}}\leq n$.
\end{proof}
\begin{proposition}
Asymptotic equivalent AS.R. spaces have the same asymptotic dimension.
\end{proposition}
\begin{proof}
Let $f:X\rightarrow Y$ and $g:Y\rightarrow X$ be two AS.R. mappings between AS.R. spaces $(X,\lambda)$ and $(Y,\lambda^{\prime})$, such that $g\circ f(A)\lambda A$ and $f\circ g(B)\lambda^{\prime}B$ for all subsets $A\subseteq X$ and $B\subseteq Y$. Suppose that $\operatorname{asdim}_{\lambda}X=n$. Let $\mathcal{U}$ be a uniformly bounded cover of $Y$ and let $g^{*}(\mathcal{U})=\{g(U)\mid U\in \mathcal{U}\}$. For all $U\in \mathcal{U}$ we have $f\circ g(U)\lambda U$ so $g(U)\subseteq f^{-1}(f\circ g(U))$ is bounded. Assume that $A,B\subseteq g(Y)$ and $A\sim_{g^{*}(\mathcal{U})}B$. Let $C=g^{-1}(A)\bigcap S_{\mathcal{U}}(g^{-1}(B))$ and $D=g^{-1}(B)\bigcap S_{\mathcal{U}}(g^{-1}(A))$. Since $A\sim_{g^{*}(\mathcal{U})}B$, it is straightforward to show that $g(C)=A$ and $g(D)=B$. We have $C\sim_{\mathcal{U}}D$ so $C\lambda^{\prime} D$. Since $g$ is an AS.R. mapping $A\lambda B$. Thus $g^{*}(\mathcal{U})$ is a uniformly bounded cover of $g(Y)$. By \ref{asdim}, $\operatorname{asdim}_{\lambda_{g(Y)}}g(Y)\leq n$. So there is a uniformly bounded cover $\mathcal{V}$ of $g(Y)$ such that $g^{*}(\mathcal{U})$ refines it and $\mu(\mathcal{V})\leq n+1$. Let $g_{*}(\mathcal{V})=\{g^{-1}(V)\mid V\in \mathcal{V}\}$. Since $g$ is an AS.R. mapping all members of $g_{*}(\mathcal{V})$ are bounded. Suppose that $M,N\subseteq Y$ and $M\sim_{g_{*}(\mathcal{V})}N$. It is easy to verify $g(M)\sim_{\mathcal{V}}g(N)$ so $g(M)\lambda g(N)$. Since $f\circ g(M)\lambda^{\prime} M$ and $f\circ g(N)\lambda^{\prime} N$ so $M\lambda^{\prime}N$. Thus $g_{*}(\mathcal{V})$ is a uniformly bounded cover of $Y$. It is straightforward to show that $\mathcal{U}$ refines $g_{*}(\mathcal{V})$ and $\mu(g_{*}(\mathcal{V}))\leq n+1$. Therefore $\operatorname{asdim}_{\lambda^{\prime}}Y\leq \operatorname{asdim}_{\lambda}X$. Similarly one can show that $\operatorname{asdim}_{\lambda}X\leq \operatorname{asdim}_{\lambda^{\prime}}Y$.
\end{proof}
\section*{ACKNOWLEDGEMENTS}
The authors wish to express their gratitude to Jesus A. Alvarez Lopez for several helpful comments.

\end{document}